\documentclass[11pt,a4paper]{article}
\usepackage{amsmath, amssymb}
\usepackage{amsthm}
\usepackage{graphicx}

\newcommand{\R}{\mathbb{R}}

\newcommand{\C}{\mathbb{C}}
\newcommand{\N}{\mathbb{N}}

\newcommand{\del}{\partial}

\newcommand{\abs}[1]{\left\lvert#1\right\rvert}

\newcommand{\dd}[2]{\frac{\del #1}{\del #2}}   

\newcommand{\define}[1]{\textsl{#1}}
\newcommand{\gap}{\vspace{2ex}}

\newcommand{\inv}{^{-1}} 

\theoremstyle{definition}
\newtheorem{thm}{Theorem}
\newtheorem{theorem}[thm]{Theorem}
\newtheorem{prop}[thm]{Proposition}

\newtheorem{lemma}[thm]{Lemma}

\newtheorem{defn}[thm]{Definition}

\newtheorem{rmk}[thm]{Remark}

\newtheorem*{rmkn}{Remark}

\newtheorem*{thmn}{Theorem}

\theoremstyle{remark}

\newcommand{\J}{\mathcal{J}}
\newcommand{\LL}{\mathbb{L}}
\newcommand{\hf}{\tfrac{1}{2}}
\newcommand{\BS}{Bohr-Sommer\-feld}
\newcommand{\lf}{\ell}
\newcommand{\sniat}{\'Sniatycki}
\newcommand{\U}{\mathfrak{U}}
\newcommand{\V}{\mathfrak{V}}

\DeclareMathOperator{\hol}{hol}

\begin{document}

\title{Geometric quantization of integrable systems
with hyperbolic singularities}

\author {Mark D. Hamilton$^1$
 and Eva Miranda $^2$}

\addtocounter{footnote}{1}
 \footnotetext{$^{}$  Graduate School of Mathematical Sciences,
University of Tokyo, 3-8-1 Komaba, Tokyo 153-8914 Japan,
\texttt{umbra@math.toronto.edu}\\ 
Supported by a PIMS Postdoctoral Fellowship and
an Oberwolfach-Leibniz Fellowship}
 \addtocounter{footnote}{1}

 \footnotetext{$^{}$ Departament de
Matem\`{a}tiques,
 Universitat Aut\`{o}noma de Bar\-celona, E-08193 Bellaterra, Spain,
\texttt{miranda@mat.uab.cat} \\  Research supported by a Juan de la
Cierva contract reference number JCI-2005-1712-18 and partially
supported by the DGICYT/FEDER, project number
 MTM2006-04353 (Geometr\'\i a Hiperb\'olica y Geometr\'\i a
 Simpl\'ectica).
 }
 \addtocounter{footnote}{2}

\date{January 21, 2009}

\maketitle

\begin{abstract}
We construct the geometric quantization of a compact surface
using a singular real polarization coming from an integrable system.
Such a polarization always has singularities, which we assume to be of
nondegenerate type.
In particular, we compute the effect of hyperbolic singularities,
which make an infinite-dimensional contribution to the
quantization,
thus showing that this quantization depends strongly on polarization.
\end{abstract}

\tableofcontents

\section{Introduction}\label{s:intro}

In the theory of geometric quantization,
the ``quantization'' of a symplectic manifold $M$
is constructed from sections of a complex line bundle over $M$.
The ingredients for geometric quantization are as follows: a symplectic
manifold $(M, \omega)$, a complex line bundle $\LL$ over $M$, and a
connection $\nabla$ on $\LL$ whose curvature is $\omega$.
We also require a  \emph{polarization,} which is an integrable
complex Lagrangian distribution
(see~\cite{woodhouse} for more information).
A \emph{real polarization} is given by a foliation of
$M$ into Lagrangian submanifolds.
If $\J$ is the sheaf of sections of $\LL$ that are
covariant constant (with respect to $\nabla$) in the directions tangent
to the leaves of the foliation,
then the
\emph{quantization} of $M$ is
\[ \mathcal{Q}(M) = \bigoplus_{k\geq 0} H^k(M;\J), \]
where $H^*(M;\J)$ is the cohomology of $M$ with coefficients
in $\J$.\footnote{Some authors,
particularly those who take an index theory approach to quantization
(e.g.~\cite{ggk})
define the quantization as the alternating sum
of cohomology groups, rather than the straightforward sum as we do here.
However, as we will show, all but one of these groups are zero,
and so it does not really matter which definition we take.
Guillemin and Sternberg in~\cite{guilleminandsternberg} avoid this question
altogether and say merely that ``the main objects of interest are
the cohomology groups  $H^k(M;\J)$.''}

The main result about quantization using real polarizations is a
theorem of \'Sniatycki~\cite{sniatpaper} from 1975: If the leaf
space $B^n$ is a manifold and the map $\pi \colon M^{2n} \to B^n$ a
fibration with compact fibres, then all of these cohomology groups
are zero except in degree $n$. Furthermore, $H^n$ can be expressed
in terms of Bohr-Sommerfeld leaves. A \emph{Bohr-Sommerfeld} leaf is
one on which is defined a global section which is flat along the
leaf (see Definition~\ref{d:bs}). The set of \BS\ leaves is
discrete, and \'Sniatycki's result says that the dimension of $H^n$
is equal to the number of Bohr-Sommerfeld leaves. (It actually
applies to non-compact manifolds as well, in which case the non-zero
cohomology is in degree equal to the rank of a fibre of $\pi$.
However, in this paper we only consider the compact case.)

The hypothesis that $B^n$ be a manifold is quite restrictive, however.
For example, in a completely integrable system, by the Arnol'd-Liouville
theorem the fibres of the moment map are generically Lagrangian tori,
but there may be fibres which have smaller dimension or are not manifolds.
This is like a real polarization except for the singularities, and so we view
it as a singular real polarization and extend the quantization machinery
to this case.

A local classification of the types of nondegenerate singularities
appearing in integrable systems has been established by Eliasson and
the second author in~\cite{eliassonthesis, eliassonelliptic, evathesis}.
It has as starting point the algebraic classification
due to Williamson~\cite{wi} of Cartan subalgebras of the Lie
algebra of the symplectic group,
and is given in terms of a local model for the
components of the moment map near the singularity.
Singularities can be written as a
product of three basic types, which are called \emph{elliptic,
hyperbolic,} and \emph{focus-focus.}

In~\cite{mhthesis}, the first author computed the quantization of
systems with only elliptic singularities.
The result obtained was similar to \sniat 's: all cohomology groups are
zero except in degree $n$, and $H^n$ has dimension equal to
the number of \BS\ leaves.
However, the singular \BS\ leaves do not make a contribution to the
cohomology and are not included in this count.

A natural question, then, is what are these cohomology groups
for a system with the other types of singularities?
This paper addresses the case of hyperbolic singularities in two dimensions.
We plan to return to the focus-focus case, and the general case of
singularities of mixed types, in a future paper.
Note that this paper completes the case of this quantization
(with respect to singular real polarizations) for compact manifolds of
two dimensions,
since focus-focus components can only appear in dimensions four or higher.

The main result of this paper (Theorem~\ref{mainthm}) is:
\begin{thmn}
Let $(M, \omega, F)$ be a two-dimensional, compact, completely integrable
system, whose moment map has only nondegenerate singularities.
Suppose $M$ has a prequantum line bundle $\LL$,
and let $\J$ be the sheaf of sections of $\LL$ flat along the leaves.
The cohomology $H^1(M,\J)$ has two
contributions of the form $\C^\N$ for each hyperbolic singularity,
each one corresponding to a space of Taylor series in one complex variable.
It also has one $\C$ term for each non-singular Bohr-Sommerfeld leaf.
That is,
\begin{equation}\label{eq:mainthmn}
H^1 (M;\J) \cong \bigoplus_{p \in \mathcal{H}}
\bigl( \C^\N \oplus \C^\N\bigr)
\oplus \bigoplus_{b\in BS} \C_b .
\end{equation}
The cohomology in other degrees is zero.
Thus, the quantization of $M$
is given by~\eqref{eq:mainthmn}.
\end{thmn}

We follow the methods of~\cite{mhthesis}, dividing the manifold up
into open sets and computing the cohomology of each set
individually, and then piecing them together using a Mayer-Vietoris
argument. The case of neighbourhoods of regular leaves is covered by
the theorems in~\cite{sniatpaper} and~\cite{mhthesis}, so we
concentrate on a neighbourhood of a singular leaf, where we compute
the cohomology groups using a \v Cech approach.\footnote{
Other authors, including \'Sniatycki~\cite{sniatpaper} 
and Rawnsley~\cite{rawnsley},
have used an approach based on an abstract de Rham theorem, 
using a resolution of the sheaf $\J$ to compute the cohomology.
One of the main issues of this approach is to prove the resolution is fine, 
which requires a Poincar\'{e} lemma adapted to the polarization
(see for instance \cite{rawnsley}). 
Such a lemma, for the case when the polarization has nondegenerate 
singularities, has been 
proved by the second author and San V\~u Ng\d oc
in \cite{evasan}; this result could be applied to prove that a similar 
resolution applies to our situation.
However, \'Sniatycki's computation in the regular case 
strongly uses the existence of action-angle coordinates in a 
neighbourhood of the whole fibre 
(although he does not use the term ``action-angle'').
When the polarization is singular, ``singular
action-angle coodinates''  do not, in general, exist in a whole
neighbourhood of the singular fibre, but only on a
neighbourhood of the singular point (see~\cite{toulet}), 
and so we would still have to divide up a neighbourhood of a singular 
leaf up into pieces, deal with each piece separately, and 
then fit them back together again.  
For this reason we find it simpler 
to just work with \v Cech cohomology directly.}

One of the issues in geometric quantization is ``independence of
polarization,'' the question of
whether different polarizations give equivalent quantizations.
When we allow singularities in the polarization, we find that the
quantization depends strongly on the polarization,
in the sense that we can easily
introduce new hyperbolic singularities by 
using surgery of integrable systems
(see~\S\ref{s:hopf}).
We also give explicit examples coming from mechanics of two different
systems on a sphere with different quantizations:
rotation about the vertical axis, and the Euler equations on the sphere.
The first one has no hyperbolic singularities, while the second one has
two, giving four infinite-dimensional contributions to the quantization.

The organization of this paper is as follows:
We review definitions and terminology in section~\ref{s:def},
and prove some properties of the sheaf of flat sections in
section~\ref{s:flatsec}.
The cohomology computation for the simplest hyperbolic system
is carried out in sections~\ref{cohom-setup} and~\ref{ugly}, and
extended to more complicated leaf structures in~\ref{s:manysing}.
In section~\ref{s:hopf} we describe the surgery of integrable
systems and give two examples from mechanics with different polarizations
having different quantizations.
Finally, section~\ref{refcov} contains a technical proof
having to do with \v Cech cohomology.

\subsection{Acknowledgements}
The acknowledgements part in this paper
deserves its own subsection. During the process of working on this
project, the authors have been substantially helped by
many people along the way.
First, a big thanks must go to Victor Guillemin
who has helped us a lot with this problem,
and has enthusiastically followed up on its progress.
We are also extremely grateful to Yael Karshon for the many helpful
conversations and suggestions during our visit to Toronto during the
early stages of this project.
Many thanks also to both Victor and Yael for the
invitations to Boston and Toronto which made an important contribution
to our work.

We are very grateful to Mathematisches Forschungsinstitut Oberwolfach for
the opportunity to work on this project in the beautiful settings of
the Institute.  Oberwolfach has provided a perfect
working atmosphere in the fantastic Black Forest, which gave a beautiful
backdrop to our sometimes messy calculations.

Thanks to Jerrold Marsden and Tudor Ratiu who kindly
provided us with the lovely picture of the Euler equations 
(Figure~\ref{euler}) in section 7.2.

Last but not least, we want to thank Roger and
Tess from 9 Baldwin Street in Toronto for their hospitality during
the early stages of this project.

\section{Definitions}\label{s:def}

\subsection{Integrable systems}
Let $(M^{2n},\omega)$ be a symplectic manifold of dimension $2n$.
The Poisson bracket is defined on $M$ by
$\{f_i,f_j\}=\omega(X_{f_i}, X_{f_j})$
where $X_{f_j}$ is the Hamiltonian vector field of $f_j$.
A \define{completely integrable system} is given by a set of $n$
functions $f_1,\dots, f_n$ which Poisson commute and which are
generically functionally independent.

Since $0 = \{f_i,f_j\} = \omega(X_{f_i}, X_{f_j})$
and $[X_{f_i},X_{f_j}] = X_{\{f_i,f_j\}} = 0$,
the distribution
generated by the Hamiltonian vector fields of the functions $f_i$ is
involutive and the regular integral manifolds are Lagrangian
submanifolds of $(M^{2n}, \omega)$.

 The collection of functions
$F=(f_1,\dots,f_n)$ is often called
the \emph{moment map} in the literature of integrable
systems.  Observe that when the manifold
is compact, the moment map $F$ has singularities, which
correspond to singularities of the distribution by
Hamiltonian vector fields. A whole theory has been developed (and is
still being developed) for the singularities of this mapping and the
symplectic invariants attached to them. In the case that the
singularities are non-degenerate (in the sense of
\cite{eliassonelliptic}), there is a local symplectic Morse theory for
these systems (see \cite{eliassonthesis} and \cite{evathesis}).

If $(M,\omega)$ is two-dimensional, a completely integrable
system is just a function $F \colon M \to \R$.  In this case,
a non-degenerate singular point $p$ is a point where $d_p F=0$ and
the Hessian $d_p^2 F$ is non-degenerate.
There are only two types of non-degenerate singularities for
integrable systems in dimension 2: \emph{hyperbolic} (when the Hessian is
indefinite) or 
\emph{elliptic} (when the Hessian is positive or negative definite).

The following theorem is due to Colin de Verdi\`ere and
Vey~\cite{colinvey}, and is a special case in two dimensions
of more general results by Eliasson and the second author 
(\cite{eliassonthesis, eliassonelliptic, evathesis}). It
gives a symplectic local model for a neighbourhood of the singularity.

\begin{theorem}\label{eliasson-coords}
Let $F\colon(M^2,\omega)\longrightarrow \mathbb R$ be a function and let
$p$  be a non-degenerate singular point of $F$. Let $Q$ be the
quadratic form corresponding to the Hessian of $F$ at $p$.

Then there exists a local diffeomorphism from a neighbourhood $Z$ of
$p$ to a neighbourhood of $0$ in $\R^2$ taking $\omega$ to the
symplectic form $dx\wedge dy$ and $F$ to a function $\phi(Q)$. If
the hessian $Q$ is positive definite the germ of the function $\phi$
characterizes the pair $(F,\omega)$. If $Q$ is not definite then the
jet at the point $p$ of the function $\phi$ characterizes the pair
$(F,\omega)$.
\end{theorem}

\begin{rmk}
As a consequence of this theorem, after putting $Q$ in a canonical form,
we can assume from now on that the foliation in a neighbourhood of a
singular point $p$ corresponding to $0 \in \R^2$ is given by
the vector field
\begin{itemize}
\item $ Y= -y\frac{\partial}{\partial x} + x\frac{\partial}{\partial y}$
when $Q=x^2 + y^2$ ($p$ is elliptic) or
\item $ Y=x\frac{\partial}{\partial x}-y\frac{\partial}{\partial y}$
when $Q=xy$ ($p$ is hyperbolic)
\end{itemize}
and the symplectic form is $\omega=dx\wedge dy$.
We call these $x$-$y$ coordinates ``Eliasson coordinates.''

In the case that $p$ is hyperbolic, we usually take $Z$ to be a 
``hyperbolic cross'' (see Figure~\ref{hypcross}), what Toulet in~\cite{toulet} 
calls an ``\'etoile canonique.''
\end{rmk}

\begin{figure}[h]
\centerline{\includegraphics{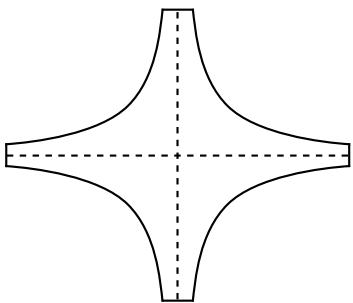}}
\caption{A ``hyperbolic cross''}\label{hypcross}
\end{figure}

\subsection{Geometric quantization}
Let $(M,\omega)$ be a symplectic manifold.
A \define{prequantization line bundle} is a complex line bundle $\LL$
over $M$, equipped with a connection $\nabla$ whose curvature is $\omega$.
A \define{real polarization} is a foliation of
$M$ into Lagrangian submanifolds.
(For a more complete description of geometric quantization,
see~\cite{woodhouse} or~\cite{sniatbook}.)

Suppose $(M,\omega,F)$ is a compact completely integrable system.
We wish to compute the quantization of $M$ using the singular
real polarization given by the singular foliation by levels of $F$,
which (as noted in the Introduction) are generically Lagrangian tori.

\begin{defn}\label{lf-flat}
A section $\sigma$ of $\LL$ is \define{flat along the leaves} or
\define{leafwise flat} if it is covariant constant along the fibres
of $F$, with respect to the prequantization connection $\nabla$.
This means that $\nabla_X \sigma = 0$ for all $X$ tangent to fibres
of $F$. Denote by $\J$ the sheaf of smooth sections which are flat
along the leaves.\footnote{
The fact that the sections are smooth is an important factor in our 
computations.  Another approach to quantization using polarizations with 
singularities would be to consider singular sections, given by distributions 
instead of smooth functions.  We hope to investigate this approach 
in a future paper.}
\end{defn}

\begin{defn}\label{d:qn}
With $(M,\omega,F)$, $\LL$, and $\J$ as above, the \define{quantization}
of $M$ is
\begin{equation*}
\mathcal{Q}(M) = \bigoplus_{k\geq 0} H^k(M;\J).
\end{equation*}
\end{defn}

\begin{rmk}\label{kost}
In the theory of geometric quantization as originally developed independently
by Kostant and Souriau,
the quantum space was the section of ``polarized'' sections of $\LL$
(which correspond to our leafwise flat sections).
However, with a real polarization with compact leaves, there are no
global polarized sections (see the proof of Theorem~\ref{higher-cohom}).
One solution to this problem, suggested by Kostant in~\cite{kostant},
is to look at higher cohomology, which is what we do.
\end{rmk}

A note on terminology: We have two rivals for the term ``flat'' in this
paper.
We will distinguish them by specifying
\emph{leafwise flat} as above,
versus
\emph{analytically flat} as follows:

\begin{defn}
A function is \define{Taylor flat} or \define{analytically flat}
(at some specified point, which is often understood)
if it vanishes to infinite order at that point,
that is, if all of its Taylor coefficients are zero.
\end{defn}

Our results will be expressed in terms of Bohr-Sommerfeld leaves.
\begin{defn}\label{d:bs}
A leaf $\lf$ of the (singular) foliation is a \define{\BS\ leaf}
if there is a leafwise flat section $\sigma$ defined over all of $\lf$.
\end{defn}
Note that, while leafwise flat sections always exist locally
(because the curvature of $\nabla$ is $\omega$, which is zero
when restricted to a leaf),
the condition of existing globally is quite strong.
The set of \BS\ leaves is discrete
(in the leaf space).
Note also that a leaf is \BS\ iff its holonomy is trivial
around all loops contained in the leaf.

\section{The leafwise flat sections}\label{s:flatsec}

We first prove several properties of elements of the sheaf $\J$,
which will be instrumental in what follows.
In particular, sections in $\J$ are analytically flat in particular
ways: see Propositions~\ref{flatatO} and~\ref{fourquad}.
For all of this section (and indeed, the rest of this paper),
$Z$ denotes the neighbourhood of the hyperbolic singular point
given in Theorem~\ref{eliasson-coords}.

\begin{lemma}\label{Thetaform}
We may choose a trivializing section of the
prequantization line bundle $\LL$ over $Z$
so that the potential one-form of the prequantum connection
is $\Theta_0 = \hf ( x\, dy - y\, dx)$ in Eliasson coordinates.
\end{lemma}

\begin{rmkn}
Remember that the potential one-form $\Theta$ of a connection,
relative to some trivialization,
is defined as follows:
If $s$ is the trivializing section, and $\sigma = \psi s$ is a section,
then
\begin{equation}\label{connection}
\nabla_X \psi s = \bigl( X(\psi) - i \psi \Theta(X) \bigr) s.
\end{equation}
\end{rmkn}

\begin{proof}[Proof of Lemma \ref{Thetaform}]
Since $Z$ is contractible,
$\LL$ is trivializable over $Z$.
Let $s$ be a trivializing section, and let $\Theta$ be the
potential one-form on $Z$ defined by~\eqref{connection}.
Since the curvature of the connection is $\omega$,
$d\Theta = \omega = d\Theta_0$, and so
$\Theta - \Theta_0$ is closed on $Z$ and, therefore, exact.
Write $dg = \Theta - \Theta_0$, and define a new trivialization $s_0$
of $\LL$ over $Z$ by $s_0 = e^{ig} s$.
Writing a section $\sigma$ as $\psi_0 s_0$, it is easy to check that
\[ \nabla_X \psi_0 s_0 = \bigl( X(\psi_0) - i \psi_0 \Theta_0(X) \bigr) s_0 \]
and so $\Theta_0$ is the potential one-form of $\nabla$ with respect to $s_0$.
\end{proof}

\begin{prop}\label{flatatO}
If $\sigma \colon Z \to \LL$ is a smooth leafwise flat section
defined over $Z$, then $\sigma$ is Taylor flat at the singular point.
That is,
\begin{equation*}
\frac{\del^{j+k}\sigma}{\del^j x\, \del^k y}\biggr\rvert_{(0,0)} = 0
\quad \text{for all } j,k
\end{equation*}
\end{prop}

\begin{proof}
According Theorem~\ref{eliasson-coords},
the foliation by level sets is generated by
\[ Y=x\frac{\partial}{\partial x}-y\frac{\partial}{\partial y}.\]

Take a trivializing section $s$ of $\LL$ as in Lemma~\ref{Thetaform},
so that the prequantum connection can be written
\[ \nabla_X(\sigma)=X(\sigma)-i \Theta_0(X) \sigma \]
(where here $\sigma$ represents a complex-valued function).
Thus any leafwise flat section $\sigma$ must satisfy the equation
\[Y(\sigma)= -ixy \sigma.\]
Writing $\sigma = \sigma_1 + i \sigma_2$ with $\sigma_1$ and $\sigma_2$
both real, we obtain the following two equations:
\begin{equation}\label{eqsig12}
\begin{split}
Y(\sigma_1)&=xy \sigma_2\\
Y(\sigma_2)&=-xy \sigma_1
\end{split}
\end{equation}

Let $\sum_{ij} a_{ij}x^i y^j$ be the Taylor expansion of $\sigma_1$
and $\sum_{ij} b_{ij}x^i y^j$ the Taylor expansion of $\sigma_2$.
We want to see that $a_{ij}=0$ and $b_{ij}=0$ $\forall i, j$.
In order to do that we plug the Taylor expansions into the
system~\eqref{eqsig12},
and obtain, for all $i,j$, the following system of equations for the
$a_{ij}$ and $b_{ij}$:
\begin{equation}\label{aijbij}
\begin{split}
(i-j) a_{ij} &=b_{i-1,j-1}\\
(i-j) b_{ij} &=a_{i-1,j-1}
\end{split}
\end{equation}

We distinguish two cases:
\begin{enumerate}
\item $i=j$.  In this case, we obtain immediately $b_{i-1,i-1}=0$
and $a_{i-1,i-1}=0$, for all $i$.

\item $i \neq j$. In this case, combining and solving
equations~\eqref{aijbij} yields:

\[b_{ij}=\frac{-b_{(i-2)(j-2)}}{(i-j)^2}.\]

Now iterating this process $k$ times we obtain:

\[ b_{ij}=\frac{(-1)^k b_{(i-2k)(j-2k)}}{(i-j)^{2k}}. \]

Now let $k$ be such that $i<2k$, then $b_{(i-2k)(j-2k)}=0$ since the
coefficients of the Taylor-Laurent expansion of a smooth function
vanish for negative subindexes.
\end{enumerate}

From this, we obtain $b_{ij}=0$ for all $i,j$ and therefore also
$a_{ij}=0$ for all $i,j$.
\end{proof}

\begin{prop}\label{p:flatsecn}
Let $U$ be an open set which does not intersect the singular leaf
and which is contained in the set $Z$ given in Theorem~\ref{eliasson-coords}.
Leafwise flat sections defined over $U$
(i.e.\ elements of $\J(U)$)
can be written in Eliasson coordinates as
\begin{equation}\label{flatsecn}
a(xy) e^{\frac{i}{2} xy \ln \abs{\frac{x}{y}}}
\end{equation}
where $a$ is a smooth complex function of one
variable.
\end{prop}

\begin{proof}
Define coordinates $(h,\beta)$ on the quadrant $\{ x>0, y>0 \}$ in $\R^2$
by
\begin{equation}\label{hbcoords}
\begin{split}
h &= xy \\
\beta &= \hf \ln \abs{\frac{x}{y}}
\end{split}
\end{equation}
so that
\begin{equation}\label{xycoords}
\begin{split}
x &= \sqrt{h}\, e^\beta\\
y &= \sqrt{h}\, e^{-\beta}
\end{split}
\end{equation}
This is valid provided neither $x$ nor $y$ is zero.
Also, $\omega = d\beta \wedge dh$
and $\Theta = -h \, d\beta$,  as can easily be checked.
Finally, in these coordinates, the vector field $Y$ is $-\dd{}{\beta}$.

Using the trivializing section from Lemma~\ref{Thetaform} we identify
a section $\sigma$ with a complex-valued function.
Then, using~\eqref{connection}, $\sigma$ will be flat if
\begin{equation}\label{flatcondn}
\nabla_{\dd{}{\beta}}\sigma = \dd{}{\beta}\sigma -
i \sigma h\, d\beta(\dd{}{\beta}) = 0
\end{equation}
which becomes
\[ \dd{\sigma}{\beta} = i \sigma h \]
which has solution
\[ \sigma = a(h) e^{ih\beta} \]
where $a$ is an arbitrary smooth (complex) function of one variable.
Changing back to $x$-$y$ coordinates gives the desired form for $\sigma$.

The above argument was valid for any open set in the first quadrant in $\R^2$.
A similar argument applies in the other quadrants,
with a slightly difference choice of signs in~\eqref{xycoords}.
(For example, in the second quadrant one should take $x=-\sqrt{-h}\,e^\beta$,
$y=\sqrt{-h}\, e^{-\beta}$.  Equations~\eqref{hbcoords} are unchanged.)
\end{proof}

\begin{prop}\label{fourquad}
Any leafwise flat section $\sigma$ defined over $Z$ can be written as
a collection
\begin{equation*}
\sigma_j = a_j(xy) e^{\frac{i}{2} xy \ln \abs{\frac{x}{y}}}\qquad j=1,2,3,4
\end{equation*}
where $a_j$ is a complex-valued smooth function of one variable, 
analytically flat at $0$, 
with domain such that $a_j(xy)$ is defined on the $j^\text{th}$
open quadrant of $\R^2$.
Conversely, given four such $a_j$, they fit together to define
a leafwise flat section $\sigma$ over $Z$ using the formula above.
\end{prop}

\begin{proof}
Suppose we are given a leafwise flat section $\sigma$.
By Proposition~\ref{p:flatsecn}, $\sigma$ has the given form on any
open set $U$ which does not intersect the axis, and in particular on the
first quadrant part of $Z$ (in Eliasson coordinates).
By Proposition~\ref{flatatO}, $\sigma$ is analytically flat
at $(0,0)$.
This implies that the function of one variable $\sigma(x,x)$
is analytically flat at $x=0$.
But $\sigma(x,x) = a_1(x^2)$ (since the logarithm term is 0 if $y=x$),
and so all (one-sided) derivatives of $a_1$ vanish at 0.
A similar argument holds for the other quadrants.

Conversely, suppose we are given $a_j$ as in the proposition, and
note that by Proposition~\ref{p:flatsecn} they define a leafwise
flat section everywhere except on the axes. On the axes, since the
functions $a_k$ are Taylor flat, the jets of the functions agree as
we approach from either side (see below), and so the four components
piece together to make a smooth section over the entire ``hyperbolic
cross'' neighbourhood $Z$.

In more detail, note first the following facts, the proofs of which
are easy exercises in first-year calculus:
\begin{lemma}
Let $a(t)$ be Taylor flat at 0.  Then
\[
\lim_{t\to 0^+} a(t) (\ln t)^n = 0
\qquad \text{and }
\lim_{t\to 0^+} a(t)r(t) = 0 \]
for all $n \in \N$ and all rational functions $r$.
\end{lemma}
\begin{proof}
L'H\^opital, plus the observation that for $0<t<1$,
$\abs{\ln t} < \abs{\frac{1}{t}}$.
\end{proof}
\begin{lemma}
If $a(t)$ is Taylor flat, then $\lim_{t\to 0^+} a(t)\, r(t, \ln t) = 0$
for all rational functions $r$.
\end{lemma}
Finally (continuing the proof of Proposition~\ref{fourquad}), any term
\[\dd{^{j+k}}{x^j\,\del y^k} a_j(xy) e^{\frac{i}{2} xy \ln \abs{\frac{x}{y}}}\]
will be the sum of terms of the form
\[ a^{(n)}(xy)\, r(x,y)\, \bigl(\ln \bigl\lvert\tfrac{x}{y}\bigr\rvert \bigr)^m
e^{\frac{i}{2} xy \ln \abs{\frac{x}{y}}},\]
and thus by the Lemma will approach $0$ as $x,y \to 0$.
\end{proof}

From this proposition we see two important facts.

\begin{rmk}
As in the regular and elliptic cases (in~\cite{mhthesis}),
leafwise flat functions have the form of a smooth function
on a transversal to the leaf, times a fixed function of the leaf variable.
This is related to parallel transport, see \S\ref{ss:pt} below.
In a way, the coordinates $(h,\beta)$ used in the proof are somewhat like
action-angle coordinates, except that they are not defined on the
singular leaf,
and $\beta$ is not an ``angle,'' but runs from $-\infty$ to $\infty$.
\end{rmk}

\begin{rmk}\label{rmk:sections}
A smooth, leafwise flat section over
a neighbourhood of the singularity has four
essentially independent components, each defined on one quadrant.
The only requirement is that each of the functions $a_k$ on the
transversals has to vanish to infinite order at the singular leaf.
To give such a section, it suffices to give four such functions $a_k$.
This will play a key role in the cohomology computation.
\end{rmk}

\begin{defn}
Henceforth, when we say a section is ``Taylor flat at the singular leaf,''
we mean that the function on the transversal defining the section
(the function $a$ above)
is Taylor flat at the singular leaf.
\end{defn}

\section{Cohomology calculation, part I: the set-up}\label{cohom-setup}

Unlike in the elliptic case, there are different possibilities for the
topology of a leaf containing a hyperbolic singular point.
We start with the simplest possibility, where there is one
singular point and the singular leaf $\lf$ has the shape of a figure-eight.
Consider a neighbourhood $U(\lf)$ of this leaf formed by unions of regular
fibres of $F$ ``on either side'' of $\lf$, shown in
Figure~\ref{model}.
(For a concrete realization of this system, imagine a torus standing
``on its end,'' like a bicycle wheel, and take $F$ to be the height
function, normalized so that the bottom of the inside hole is at height 0.
Then $F\inv\bigl((-c,c)\bigr)$ looks like Figure~\ref{model}.)
We will carry out the computations for this example
in some detail, as it exhibits the main features we find in general.
In~\S\ref{s:manysing}
we show how these results extend to the case of more complicated
leaves, with more singularities.

\begin{figure}[h]
\centerline{\includegraphics[height=2in]{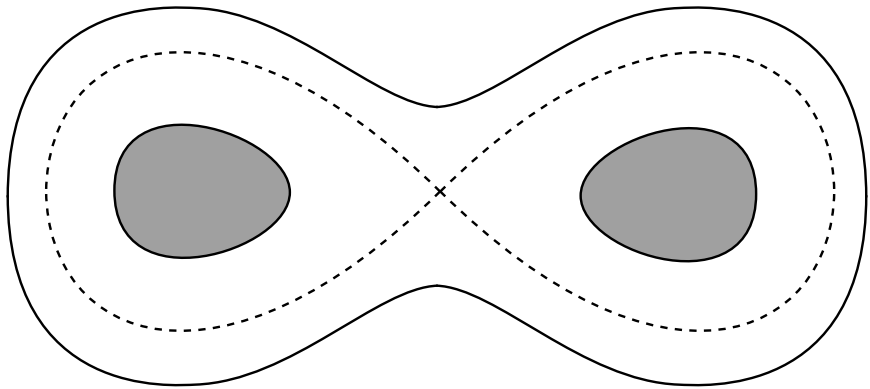}}
\caption{The model system}\label{model}
\end{figure}

We compute cohomology using a \v Cech approach,
by choosing an open covering, functions on the sets in the covering,
and so on.
Although \v Cech cohomology is defined as the direct limit over
the set of all coverings, in~(\cite{mhthesis}, \S3) we saw that the interesting
features of the cohomology appeared already in the computation using the
simplest covering, and so this is what we use here.
In \S\ref{refcov} we will
show that we have computed the actual sheaf cohomology.

Also, by comparison to~\cite{sniatpaper} and~\cite{mhthesis},
we expect the cohomology to only be non-trivial in degree 1, and
so from now on when we say ``cohomology'' without other specification,
we mean first cohomology.
In section~\ref{s:higher-cohom} we show that other cohomology is trivial.

The simple covering consists of three sets, the ``hyperbolic cross'' $Z$
together with two other sets covering the rest of $U(\lf)$, as shown in
Figure~\ref{ubiq}:

\begin{figure}[h]
\centerline{\includegraphics[height=2in]{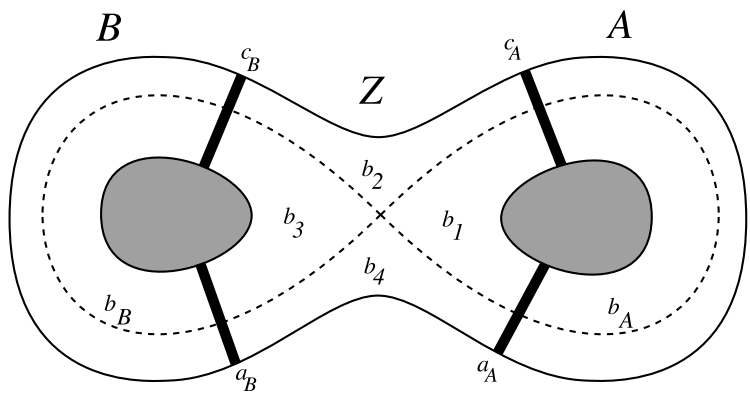}}
\caption{A simple cover of the figure-eight}\label{ubiq}
\end{figure}

The thick lines indicate an overlap of open sets,
denoted by $AZ$ for $A\cap Z$, etc, and
the dotted line indicates the singular leaf.
The letters $a$, $b$, and $c$
indicate leafwise flat sections defined on particular sets,
the collection of which determines an element of \v Cech cohomology.
We use the convention that $a$ and $c$ are functions on
the intersections of two sets, and $b$ are functions on one single set.
So, for example, $a_A$ is a function defined on $AZ = A\cap Z$, and
$b_A$ is a function defined on $A$.
(Since all overlaps include $Z$, we use $a_A$ rather than $a_{AZ}$, for
simplicity.)
Also, $b^1_Z$ through $b^4_Z$ are the sections defined on the quadrants of
the hyperbolic cross, making up a leafwise flat section over the cross
as in Proposition~\ref{fourquad}.
Thus, the $a$'s and $c$'s make up a \v Cech 1-cochain,
and the $b$'s make up a 0-cochain.
Use the ordering convention on the coboundary operator
that $(\delta b)$ on $AZ$ is $b_Z - b_A$, and on $BZ$ is $b_Z - b_B$.

We are interested in $H^1$, and so we are asking:
Given $a$'s and $c$'s as in the diagram, which define a 1-cocycle,
when do there exist $b$'s so that the coboundary of the $b$ cochain
equals the cocycle defined by the $a$'s and $c$'s?

\subsection{Parallel transport}\label{ss:pt}
In order to compare the values of sections at different points,
we use parallel transport.

Given the value of a section at one point $x_0$ on a leaf, the value
on the rest of the leaf is determined by the condition that the section
be leafwise flat.
Given a value for $\sigma(x_0)$, we can construct
a leafwise flat $\sigma$ over
the entire leaf through $x_0$ by parallel transporting $\sigma(x_0)$
along the leaf.

Given two points $P$ and $Q$ in the same leaf, we will denote
parallel transport from $P$ to $Q$ by $\tau_{PQ}$.
Thus, if $a$ is a flat section, $a(Q) = a(P) \tau_{PQ}$.
Formally, $\tau_{PQ}$ is an automorphism of $\LL_P$ to $\LL_Q$;
if $\LL$ can be trivialized over a set containing both $P$ and $Q$,
then we can just think of $\tau_{PQ}$ as a nonzero complex number.
Note that $\tau_{PQ} = \tau_{QP}\inv$, whether as automorphisms
or as complex numbers.

This is related to the description of the sections in
Proposition~\ref{p:flatsecn} as $a(h) e^{-ih\beta}$.
For a fixed value of $h$, say $h_0$, once we know $a(h_0)$,
then the value of the section is fixed everywhere on the leaf.
The term $e^{-ih\beta}$ represents the change due to parallel transport.

\section{Cohomology calculation, part II: explicit calculation}\label{ugly}

To carry out the computations, we refer to Figure~\ref{ubiq2}.

\begin{figure}[h]
\centerline{\includegraphics[height=2in]{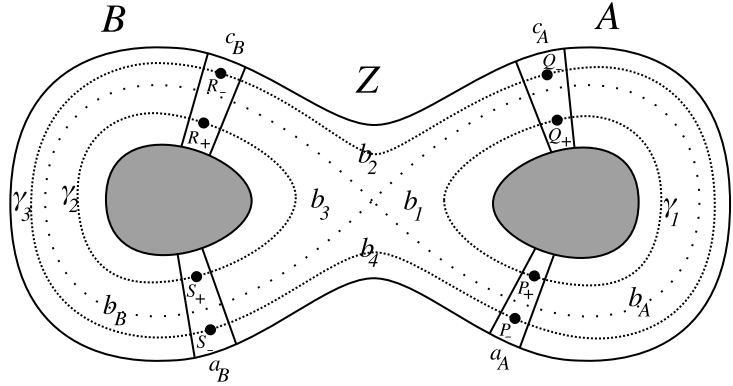}}
\caption{Diagram for the cohomology calculation}\label{ubiq2}
\end{figure}

This is the same as Figure \ref{ubiq} with more information: we have
added three leaves we will be considering, labelled $\gamma_1$,
$\gamma_2$, and $\gamma_3$, and marked points on these leaves as
shown. We have also shown the overlaps between sets, so for example
$P_+$ and $P_-$ are in $A\cap Z$.

First, we fix some notation:
$\tau^Z_{PQ}$ will denote parallel transport from $P$ to $Q$
\emph{through the set $Z$.}
When we are looking at, for example, transport from $P_+$ to $Q_+$
we will write $\tau^Z_{PQ_+}$ (rather than $\tau^Z_{P_+Q_+}$).
There will never be parallel transport between a ``$-$'' point 
and a ``$+$'' point, because they are on different leaves.

Note that an appropriate combination of parallel transport gives us
holonomy: for example,
\[ \tau_{PQ_+}^A \tau_{QP_+}^Z = \hol_{\gamma_1}. \]

\gap

Remember that what we are trying to do, exactly, is to
answer the following question:
Given a collection $\{ a_A, c_A, a_B, c_B\}$ of sections defining
a 1-cochain, when can we find sections
$\{ b_A, b_B, b_1, b_2, b_3, b_4 \}$
making up  a 0-cochain whose coboundary is the given 1-cochain?
The set of $a$'s and $c$'s where this is possible gives us $B^1$, 
the set of 1-coboundaries.
As it turns out, the three loops $\gamma_1$, $\gamma_2$, and $\gamma_3$ 
give independent contributions to the cohomology, and 
$H^1(U(\lf);\J)$ will be the direct sum of the contributions from each loop.
We look at each one in turn and collect the results together 
in Theorem~\ref{thm:summary}.


\subsection{Gamma 1}\label{ss:gamma1}

First look at $\gamma_1$.  We have the following relations
coming from the coboundary conditions:
\begin{subequations}\label{A}
\begin{align}
a_A(P_+) &= b_1(P_+) - b_A(P_+)\label{A1}\\
c_A(Q_+) &= b_1(Q_+) - b_A(Q_+)\label{A2}
\end{align}
\end{subequations}
We also have the following relations between the values of the
sections at different points:
\begin{subequations}\label{Atransport}
\begin{align}
b_A(Q_+) &= b_A(P_+)\tau^A_{PQ_+}\\
b_1(P_+) &= b_1(Q_+)\tau^Z_{QP_+}
\end{align}
\end{subequations}
From \eqref{Atransport}, \eqref{A2} becomes
\begin{equation*}
c_A(Q_+) = b_1(P_+)(\tau^Z_{QP_+})\inv -
    b_A(P_+)\tau^A_{PQ_+}
\end{equation*}
so that the system \eqref{A} becomes
\begin{subequations}\label{Arev}
\begin{align}
a_A(P_+) &= b_1(P_+) - b_A(P_+)\\
c_A(Q_+)\tau^Z_{QP_+} &= b_1(P_+) -
    b_A(P_+)\tau^A_{PQ_+}\tau^Z_{QP_+}
\end{align}
\end{subequations}
This can be viewed as a system of two equations for the two
unknowns $b_1(P_+)$, $b_A(P_+)$.  The coefficient matrix of this system
is
\begin{equation*}
\begin{bmatrix}
1&  -1\\
1&  -\tau^A_{PQ_+}\tau^Z_{QP_+}
\end{bmatrix}
\end{equation*}
which has determinant
\begin{equation*}
1- \tau^A_{PQ_+}\tau^Z_{QP_+}  = 1- \hol_{\gamma_1}.
\end{equation*}
Thus this matrix is nonsingular, and so~\eqref{A} has a unique solution,
precisely when $\hol_{\gamma_1} \neq 1$.
This solution is:
\begin{subequations}\label{hol1}
\begin{align}
a_A(P_+)  - c_A(Q_+)\tau^Z_{QP_+}   =
b_A(P_+)\bigl( \hol_{\gamma_1} - 1 \bigr)\label{hol1-1}\\
a_A(P_+) - c_A(Q_+)(\tau^A_{PQ_+})\inv
= b_1(P_+)\bigl( 1-\hol_{\gamma_1}\bigr)\label{hol1-2}
\end{align}
\end{subequations}
This gives $b_1$ and $b_A$ at the single point $P_+$; however,
as noted previously, the value of a flat section
at one point determines the value everywhere else along the leaf,
and so this gives a solution for $b_A$ and $b_1$ on the entire leaf.
Finally, by letting $P_+$ vary along a transversal to the leaves,
we get $b_A$ and $b_1$ on the entire neighbourhood inside the singular leaf.

If $\hol_{\gamma_1} = 1$, then a linear algebra argument shows
that~\eqref{A} has a solution (and thus the cocycle is a coboundary)
if{f} $a_A(P_+) - c_A(Q_+)\tau^Z_{QP_+} =0$.
Since a cocycle is defined by two smooth functions on the transversal
(determining the sections $a_A$ and $c_A$), the Bohr-Sommerfeld 
contribution to the cohomology from $\gamma_1$ is
\begin{equation*}
\frac{\{\text{cocycles}\}}{\{\text{coboundaries}\}}
\cong \frac{\{(a,c) \text{ smooth functions on an interval} \}}
{ \{ a = c \text{ at one point} \}}
\end{equation*}
This is exactly what we saw appearing in~\cite{mhthesis} (\S3.2.2).
As we saw there (Lemma 3.3), the above quotient is isomorphic to $\C$, and so
if $\gamma_1$ is Bohr-Sommerfeld,
it gives a one-dimensional contribution to cohomology.
(See also Theorem~\ref{thm:summary}.)

However, this is not the only contribution from $\gamma_1$.
The flatness
properties discussed in~\S\ref{s:flatsec} affect the calculation as well:
in searching for solutions to~\eqref{A}, we do not have complete freedom
in choosing $b_A$ and $b_1$, because of the condition that $b_1$ has to
be Taylor flat at the singular leaf.

Consider again the system~\eqref{hol1}.
It is valid for all $P_+$ inside the singular leaf.
If we think of the sections as functions of one variable as $P_+$ 
varies along a transversal to the leaf, then the properties 
discussed in~\S\ref{s:flatsec} imply that 
$b_1(P_+)$, and thus
the right-hand side of~\eqref{hol1-2}, is
Taylor flat as $P_+$ approaches the singular leaf.
Therefore, in order for the system~\eqref{hol1} to have a solution, 
it is necessary that 
\begin{equation}\label{holcon1}
a_A(P_+) - c_A(Q_+)(\tau^A_{PQ_+})\inv
\end{equation}
be Taylor flat at the singular leaf 
(viewing $P_+$ as a variable, which determines $Q_+$), 
which is to say that $a_A$ and $c_A(Q_+)(\tau^A_{PQ_+})\inv$
agree to infinite order at the singular leaf.  This will give another
contribution in cohomology, which we will clarify in~\S\ref{ss:flatfn} 
and~\ref{ss:main}.

\subsection{Gamma 2}\label{ss:gamma2}

The picture is similar for $\gamma_2$ as for $\gamma_1$.
The coboundary equations
\begin{align*}
a_B(S_+) &= b_3(S_+) - b_B(S_+)\\
c_B(R_+) &= b_3(R_+) - b_B(R_+)
\end{align*}
are exactly the same as system~\eqref{A}, with $A$ replaced by $B$,
$P$ replaced by $S$, $Q$ replaced by $R$, and $b_1$ replaced by $b_3$.
Thus, they have solutions identical to~\eqref{hol1}
with these same replacements, namely:
\begin{align}\label{hol2}
a_B(S_+)  - c_B(R_+)\tau^Z_{RS_+}
&=b_B(S_+)\bigl( \hol_{\gamma_1} - 1 \bigr)\\
a_B(S_+) - c_B(R_+)(\tau^B_{SR_+})\inv
&= b_3(S_+)\bigl( 1-\hol_{\gamma_1}\bigr)
\end{align}
The second equation gives us, by the same argument, the condition that
\begin{equation}\label{holcon2}
a_B(S_+) - c_B(R_+)(\tau^B_{SR_+})\inv
\end{equation}
has to be Taylor flat at the singular leaf. 
This gives another ``flat functions'' 
contribution to the cohomology, which we will 
discuss in~\S\ref{ss:main}.

\subsection{Gamma 3}
The computation for $\gamma_3$ is similar, except that since
$\gamma_3$ passes through all four components of $AZ$ and $BZ$, we
have four equations instead of two. We get a similar phenomenon
involving the holonomy, giving us a contribution of $\C$ for each
Bohr-Sommerfeld leaf. Since the calculation is similar to (although
longer than) the previous two, and since our main interest at the
moment is in the flat functions and the infinite contributions to
cohomology, we will leave out the \BS\ calculation, except to note
in passing that the \BS\ contribution will come from the holonomy
all around $\gamma_3$, but will still give one factor of $\C$ in
cohomology. Thus, the \BS\ leaves inside and outside the singular
leaf make equal contributions to the cohomology. For example, in the
torus realization mentioned in section~\ref{cohom-setup}, even if
the level set has two connected components (represented in
Figure~\ref{model} by the inner circles), each component is
independent in terms of its cohomology. (In fact, since these leaves
are regular, \'Sniatycki's results apply to give us their
contribution to cohomology directly.)

We focus now on the question of role of the flat functions for $\gamma_3$.
As there are two Taylor flat functions in this computation, $b_4$ and $b_2$,
we wish to find the solutions to the coboundary equations for $b_4$ and $b_2$,
which will give us two conditions that certain
combinations of the $a$'s and $c$'s have to be Taylor flat.
The calculations are similar in form, though more complicated, to those
given in~\ref{ss:gamma1}.
Out of compassion for the reader, we omit the details,
and merely give the results.

We start with
four equations coming from the coboundary conditions,
starting at $P_-$:
\begin{subequations}\label{outer-cob}
\begin{align}
a_A(P_-) &= b_4(P_-) - b_A(P_-)\label{oc1}\\
c_A(Q_-) &= b_2(Q_-) - b_A(Q_-)\label{oc2}\\
c_B(R_-) &= b_2(R_-) - b_B(R_-)\label{oc3}\\
a_B(S_-) &= b_4(S_-) - b_B(S_-)\label{oc4}
\end{align}
\end{subequations}
We also have relationships between the values of each function
at different points, coming from parallel transport:
\begin{subequations}\label{outer-transport}
\begin{align}
b_A(Q_-) &= b_A(P_-)\tau^A_{PQ_-}\label{ot1}\\
b_2(R_-) &= b_2(Q_-)\tau^Z_{QR_-}\label{ot2}\\
b_B(S_-) &= b_B(R_-)\tau^B_{RS_-}\label{ot3}\\
b_4(P_-) &= b_4(S_-)\tau^Z_{SP_-}\label{ot4}
\end{align}
\end{subequations}
Starting at $P_-$, we can use these formulae to ``push along'' the leaf
until we come back around to $P_-$.
The calculation involving $b_4$ (details omitted) yields
\begin{multline}\label{preb4}
b_4(P_-) - b_4(P_-)\tau^A_{PQ_-}\tau^Z_{QR_-}\tau^B_{RS_-}\tau^Z_{SP_-}
= a_B(S_-)\tau^Z_{SP_-}\\
+ c_A(Q_-)\tau^Z_{QR_-} \tau^B_{RS_-}\tau^Z_{SP_-}
 - a_A(P_-)\tau^A_{PQ_-}\tau^Z_{QR_-} \tau^B_{RS_-}\tau^Z_{SP_-}
 - c_B(R_-)\tau^B_{RS_-}\tau^Z_{SP_-}
\end{multline}
We recognize the coefficient of the second $b_4$ as the holonomy
(it also appears with $a_A$),
and so
this simplifies to
\begin{equation}\label{hol3b4}
\begin{split}
b_4(P_-)\bigl( 1 - \hol_{\gamma_3} \bigr)
= a_B(S_-)\tau^Z_{SP_-}
+ c_A(Q_-)\tau^Z_{QR_-} \tau^B_{RS_-}\tau^Z_{SP_-} \\
 - a_A(P_-)\hol_{\gamma_3} 
 - c_B(R_-)\tau^B_{RS_-}\tau^Z_{SP_-}.
\end{split}
\end{equation}
As in the previous sections, this tells us that this
particular combination of $a$'s and $c$'s has to
be Taylor flat at the singular leaf in order for the cohomology
equations to have a solution.

At first, this seems like another condition, which will give another
contribution to the cohomology.
However, if we look more closely, we see that it is not independent of
our earlier conditions.
Explicitly, if we take~\eqref{holcon1} times $\hol_{\gamma_3}$
plus~\eqref{holcon2} times $-\tau^Z_{SP_+}$, we obtain exactly
the right-hand side of~\eqref{hol3b4}, except
the points have $+$'s instead of $-$'s.
However, the condition applies \emph{at the singular leaf.} 
Since $P_+$ and $P_-$ approach the same point on the singular leaf, 
and since the Taylor series of a function is the same ``from either side,''
the condition in~\eqref{hol3b4} is already implied by conditions coming 
from~\eqref{hol1} and~\eqref{hol2}, and so does not give any new contribution
to the cohomology.

Similarly, we can go through the same process to solve~\eqref{outer-cob}
for $b_2$, which gives us another combination of $a$s and $c$s
that has to be Taylor flat at the singular leaf,
but which also turns out to be
already implied by~\eqref{hol1} and~\eqref{hol2}.

\subsection{The ``flat functions'' contribution to cohomology}
\label{ss:flatfn}

So far we have found two independent conditions 
\eqref{holcon1} and \eqref{holcon2} 
that certain combinations of sections must be analytically flat 
(as well as two similar conditions that turn out not to be independent).  
In this section we explore  what contributions these conditions 
make to the cohomology.
In both cases, the condition requires that two sections agree to 
infinite order at the singular leaf, which is equivalent to the condition 
that two functions of one variable (on a transversal to the leaf, 
defining the section) agree to infinite order at one point.
Let $I$ be an open interval, and fix a reference point $x_0 \in I$.
For two functions $a, c \in C^\infty(I)$,
let $a \approx c$ mean that $a$ and $c$ agree to infinite order at $x_0$.
Since each section is defined by a function on a transversal to the leaves, 
and the coboundaries are those where the two functions
agree to infinite order at the singular leaf, 
we will be looking at quotients of the form 
$C^\infty(I)^2 / \{ a \approx c \}$.

\begin{lemma}\label{CN}
The quotient $C^\infty(I)^2 / \{ a \approx c \}$ is isomorphic to 
the space of complex-valued sequences, which we denote by $\C^\N$.
\end{lemma}

\begin{proof}
Let $a^{(n)}(x_0)$ denote the $n^\text{th}$ Taylor coefficient of $a$ at $x_0$.
Then map $C^\infty(I)^2$ to $\C^\N$ by the map that puts
$\bigl(a^{(n)}(x_0)-c^{(n)}(x_0)\bigr)$ in the $n^\text{th}$ place.
This map has kernel exactly $\{ a\approx c\}$. 
To see it is surjective we 
apply  Borel's theorem which says that given a sequence $z_n$ of
complex numbers there exists a complex smooth function
$f$ such that $f^{(p)}(x_0)=z_p$. (See for example~\cite{borel2} or
\cite{borel}.) 
\end{proof}

\subsection{If the singular leaf is \BS}\label{ss:sing-BS}

So far, we have only considered the possibility of non-singular \BS\
leaves. What happens if the singular leaf is \BS? 
(Note that each of
the two loops in the singular leaf can be \BS, and that these
conditions are independent.)

Look at the system~\eqref{hol1},
which we reproduce here, and consider what happens as $\gamma_1$
approaches the singular leaf.
\begin{align*}
a_A(P_+)  - c_A(Q_+)\tau^Z_{QP_+}   =
b_A(P_+)\bigl( \hol_{\gamma_1} - 1 \bigr)\\
a_A(P_+) - c_A(Q_+)(\tau^A_{PQ_+})\inv
= b_1(P_+)\bigl( 1-\hol_{\gamma_1}\bigr)
\end{align*}
The holonomy $\hol_{\gamma_1}$ will be a smooth function of the
``leaf variable,''
and so we can look at each side of, say, the first equation above
as a function of the ``leaf variable.''
Even if the holonomy at the singular leaf is 1, so that the right side
vanishes at the singular leaf, the right side as a function
already vanishes \emph{to infinite order} at the singular leaf.
Thus the left side still has to be Taylor flat,
and so we still get the infinite-dimensional contribution to cohomology,
regardless of whether the singular leaf is \BS\ or not.

On the other hand, the contribution of one factor of $\C$ for
a regular \BS\ leaf does \emph{not} occur for the singular leaf.
This factor comes out of the cohomology calculation because of a condition
that the \emph{values} of $a$ and $c$ at the \BS\ leaf have to agree,
but this is already required by the condition that they have to agree
to infinite order.
Thus there is no additional \BS\ contribution.

\subsection{Summary of the calculations}\label{ss:main}

Here we collect the results from the preceding calculations into 
one place.

\begin{thm}\label{thm:summary}
The first cohomology of the neighbourhood $U(\lf)$ of the
figure-eight hyperbolic system given in Figure~\ref{model} has
two contributions of the form $\C^\N$, each one corresponding
to a space of Taylor series in a complex variable. It also has one
$\C$ term for each non-singular Bohr-Sommerfeld leaf. That is,
\begin{equation}\label{cohom}
H^1\bigl( U(\lf),\J\bigr) \cong \C^\N \oplus \C^\N 
\oplus \bigoplus_{b\in BS} \C_b 
\end{equation}
where the sum is over the non-singular \BS\ leaves.
\end{thm}

\begin{proof}
Let the points $P$, $Q$,
$R$ and $S$ be the points on the singular leaf that are the
limits of $P_+$, $P_{-}$, $Q_{+}$,$Q_{-}$, etc.\ 
when $\gamma_1$, $\gamma_2$ and $\gamma_3$ approach
the singular leaf. As we pointed out in the computations involving
$\gamma_1$ and $\gamma_2$, the expressions

\begin{equation*}
a_A(P_+) - c_A(Q_+)(\tau^A_{PQ_+})\inv
\end{equation*}
(equation \eqref{holcon1}) and

\begin{equation*}
a_B(S_+) - c_B(R_+)(\tau^B_{SR_+})\inv
\end{equation*}
(equation \eqref{holcon2}) can be seen as functions in the variables
$P_+$ and $S_+$ respectively (since the variables $Q_+$ and $R_+$
can be determined from these). These functions can be seen as
functions on the two transversals at $P$ and $S$ to the singular leaf. 
Thus we can think of these functions as functions of 
one variable (on an open interval $I$ centered at zero), 
which we denote by $a_A - c_A \tau^A$ and $a_B - c_B \tau^B$,
respectively. 
As in Lemma~\ref{CN}, let $f^{(n)}(0)$ denote the $n^\text{th}$
Taylor coefficient of the function $f$ at $0$.

As noted at the beginning of~\S\ref{ugly}, 
the space $Z^1$ of 1-cocycles is the collection 
$Z^1 =\{ (a_A, c_A, a_B, c_B) \}$.  
Map $Z^1$ into the right-hand side of~\eqref{cohom} as follows:
\begin{itemize}
\item Map $(c_A, a_A,c_B, a_B)$ to $(a_A - c_A\tau^A)^{(n)}(0)$ 
in the $n^\text{th}$ term of the first $\C^\N$ factor, and 
\item $(c_A, a_A,c_B, a_B)$ to $(a_B - c_B\tau^B)^{(n)}(0)$ 
in the $n^\text{th}$ term of the second $\C^\N$ factor; also,
\item for each non-singular Bohr-Sommerfeld leaf, passing through points 
$P_j$ and $Q_j$, map 
$(a_A, c_A, a_B, c_B)$ to $a_A(P_j) - c_A(Q_j)\tau^Z_{Q_jP_j}$
in the $\C$ component corresponding to that \BS\ leaf.
\end{itemize}
From the preceding discussion, the kernel of this map is precisely the 
set of coboundaries, as follows.
From~\S\ref{ss:gamma1}, if the cocycle is a coboundary then $a_A - c_A \tau_A$ 
is Taylor flat at the singular leaf (equation~\eqref{holcon1}).
From~\S\ref{ss:gamma2}, \eqref{holcon2}, 
we have the same for $a_B - c_B \tau^B$.
And finally, for each regular \BS\ leaf, being a coboundary requires 
that the values of the corresponding $a$ and $c\tau$ functions agree 
on that leaf.  Conversely, if all these conditions hold, the collection 
$(c_A, a_A,c_B, a_B)$ defines a coboundary.  
Thus, the kernel of this map is the set of coboundaries.

On the $\C^\N$ components, this map is the same map as was used in 
the proof of Lemma 17, 
which was shown there to be surjective, and so this map is surjective 
onto the $\C^\N$ components.
It is also surjective on the $\C$ components: 
the $(a-c\tau)^{(n)}(0)$ determine the jet of the functions at the origin, 
but not their values at any point away from the origin.  
Since $a_A, c_A$, etc.\ can be any smooth functions, it is easy to 
choose them so that $a_A(P_j) - c_A(Q_j)\tau_{P_j Q_j}$ 
has any desired value.
Thus, the map is surjective onto the $\C$ components.

Finally, if the singular leaf is \BS, it is excluded from the sum 
by~\S\ref{ss:sing-BS}.

Therefore we have a surjective map from the space of cocycles to the 
right side of~\eqref{cohom} whose kernel is the space of coboundaries, 
and so the cohomology is as claimed.
\end{proof}

\begin{rmk}
We can make the infinite-dimensional cohomology look slightly more
natural by viewing it as a graded vector space. Following the ideas
of the Arnol'd school around singularity theory (see for
example~\cite{arnold}), it is possible to define a filtration on the
sheaf $\J$ by letting $\J_k$ consist of solutions up to order $k$ of
the leafwise flat sections equation. This induces a grading on the
cohomology, so that the $\C^\N$ term has one $\C$ in each degree.
\end{rmk}

\subsection{Cohomology in other degrees}\label{s:higher-cohom}

So far we have been concerned with the cohomology in degree 1.
We now briefly dispose of the other degrees.

\begin{thm}\label{higher-cohom}
Let $(U(\lf), \omega, F)$, $\LL$, and $\J$ be as above.
Then the cohomology groups 
$H^k\bigl( U(\lf),\J\bigr)$ 
are zero for $k\neq1$.\\
\end{thm}

\begin{proof}
This is immediate.  First, $H^0\bigl( U(\lf),\J\bigr)$ 
is the set of global leafwise
flat sections of $\LL$. Any such section is zero except on the \BS\
leaves; since the set of \BS\ leaves is discrete, the entire section
must be zero by continuity, and so $H^0\bigl( U(\lf),\J\bigr)=0$. 
Higher cohomology groups are trivial because 
there are no triple or higher intersections in the cover. 
\end{proof}

\begin{rmkn}
Although we have computed the cohomology with respect to a certain cover
(and this is particularly evident here), 
we show in \S\ref{refcov} that it is isomorphic to the actual sheaf cohomology.
\end{rmkn}

\section{More than one singular point}\label{s:manysing}

Thus far, all of the calculations have been for the simplest system with
a hyperbolic singularity, given in Figure~\ref{model}.
In this section we perform the calculations for more complicated systems.

\subsection{The next simplest examples}

In the case where there is more than one hyperbolic singular point on
the same leaf, there are many different possibilities for the topology
of the leaf.
Two examples are the ``triple-eight'' with three loops
and the ``double-lung'' systems, each with two hyperbolic singularities,
shown below in Figures~\ref{triple8} and~\ref{lungs}.
Bolsinov and Fomenko in~\cite{bolsinovandfomenko}
give a classification of the possible topological types of leaves.

\begin{figure}[h]
\centerline{\includegraphics[height=3cm]{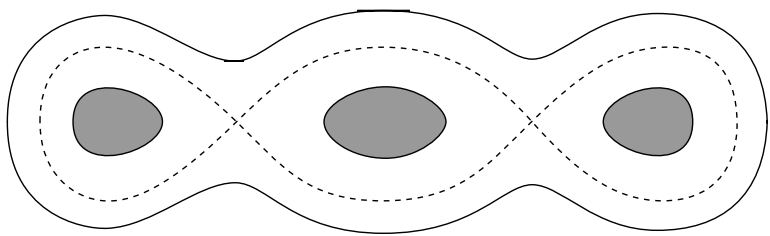}}
\caption{The ``triple-eight'' system}\label{triple8}
\end{figure}

\begin{figure}[h]
\centerline{\includegraphics[height=2in]{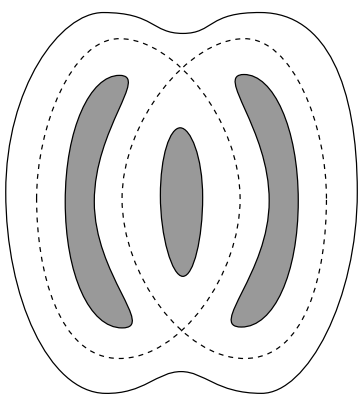}}
\caption{The ``double-lung'' system}\label{lungs}
\end{figure}

We first consider the ``triple-eight'' and
carry out the cohomology calculation for this system.
Loops around either of the two outside ``holes,'' inside the
singular leaf,
will clearly give identical calculations as in the figure-eight case,
and so we do not repeat them.
The computation for the middle loop $\gamma_4$,
shown in Figure~\ref{triple8mid}, is a bit different.

\begin{figure}[h]
\centerline{\includegraphics[height=6cm]{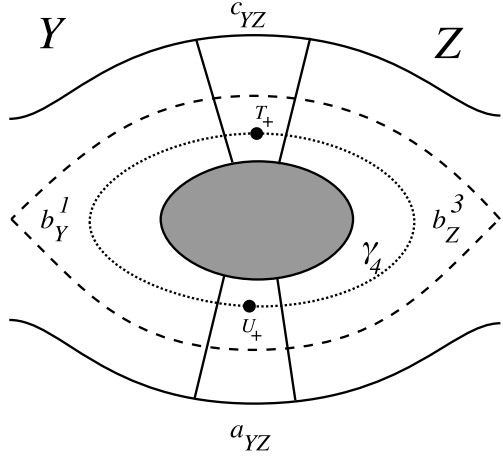}}
\caption{Closeup of the centre loop}\label{triple8mid}
\end{figure}

From $YZ$ we get the two equations
\begin{subequations}\label{YZ-coc}
\begin{align}
a_{YZ}(T_+) = b_Z^3(T) - b_Y^1(T)\label{YZ-coc1}\\
c_{YZ}(U) = b_Z^3(U) - b_Y^1(U)\label{YZ-coc2}
\end{align}
\end{subequations}
and from parallel transport we get
\begin{subequations}\label{YZ-pt}
\begin{align}
b_Z^3(U) = b_Z^3(T) \tau^Z_{TU}\label{YZ-pt1}\\
b_Y^1(T) = b_Y^1(U) \tau^Y_{UT}\label{YZ-pt2}
\end{align}
\end{subequations}
Use \eqref{YZ-pt} to write \eqref{YZ-coc2} at $T$:
\begin{equation*}
c_{YZ}(U) = b_Z^3(T) \tau^Z_{TU} - b_Y^1(T) \tau^Y_{TU}
\end{equation*}
so
\begin{equation*}
c_{YZ}(U)\tau^Y_{UT} = b_Z^3(T) \tau^Z_{TU}\tau^Y_{UT} - b_Y^1(T)
\end{equation*}
Subtracting this from \eqref{YZ-coc1}, we get
\begin{equation}\label{gamma4bohrs}
\begin{split}
a_{YZ}(T) - c_{YZ}(U)\tau^Y_{UT}
&= b_Z^3(T) - b_Y^1(T)
- b_Z^3(T) \tau^Z_{TU}\tau^Y_{UT} + b_Y^1(T) \\
&= b_Z^3(T) \bigl( 1-\hol_{\gamma_4}\bigr)
\end{split}
\end{equation}

A similar procedure gives us a solution for $b^1_Y$ in terms of
$a_{YZ}$ and $c_{YZ}$, also involving a holonomy term.
These are the familiar equations involving holonomy, which give us the
contribution due to a Bohr-Sommerfeld leaf.

However, the more interesting part is the contribution coming from
flat functions, for which we don't even need the calculation leading
to~\eqref{gamma4bohrs}, but we can see directly from~\eqref{YZ-coc}.
Since $b_Y^1$ and $b_Z^3$ are both Taylor flat at the singular leaf,
the right-hand sides of~\eqref{YZ-coc1} and~\eqref{YZ-coc2} both
vanish to infinite order at the singular leaf, and so in order
for~\eqref{YZ-coc} to have a solution, it is necessary that
\emph{both} $a_{YZ}$ and $c_{YZ}$ vanish to infinite order at the
singular leaf as well. Thus, this piece gives a contribution to the
cohomology that looks like
\[ \{ \text{smooth functions} \}^2 / \{ a \approx 0,\, c\approx 0 \} \]
namely, $(\C^\N)^2$.
Together with the two contributions coming from the loops around the
outer ``holes,'' each of which will be one contribution of $\C^\N$,
we see that there are a total of \emph{four} $\C^\N$ contributions
from the pair of singularities.

We leave as an exercise for the reader to set up and carry out
the computations for the
``double lung'' system in Figure~\ref{lungs},
and show that it also has four $\C^\N$ components in the cohomology.

\subsection{The general case}\label{ss:generalcase}

Here we show that, in general, we get two $\C^\N$ contributions to the
cohomology for each hyperbolic singular point.

Consider a covering of a neighbourhood of the singular leaf
by overlapping rectangles together
with hyperbolic crosses, as illustrated in Figure~\ref{leafcov}
(and as used in \S\ref{refcov}).
Near a hyperbolic singular point, the system looks like
Figure~\ref{hypcov}.

\begin{figure}[h]
\centerline{\includegraphics{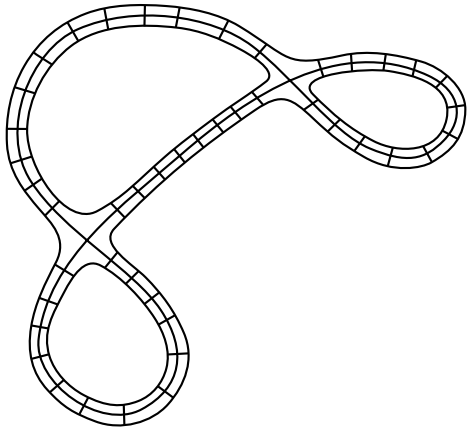}}
\caption{The covering of the leaf}\label{leafcov}
\end{figure}

\begin{figure}[h]
\centerline{\includegraphics{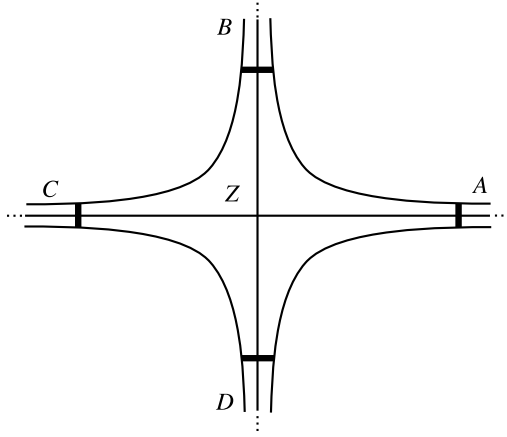}}
\caption{The covering near one hyperbolic singular point}\label{hypcov}
\end{figure}

Consider the part of the leaf passing through the
set labelled $A$ in Figure~\ref{hypcov}.
If we continue along this leaf, we will pass through a number of other
rectangles, each with their own functions defined on them and on the
corresponding intersections, and eventually reach another hyperbolic
cross (possibly the same one on a different branch).
See Figure~\ref{crosstocross}, where the $a$'s denote elements of $\J$
on double intersections (part of a 1-cochain),
and the $b$'s denote elements on the sets (part of a 0-cochain).

\begin{figure}[h]
\centerline{\includegraphics{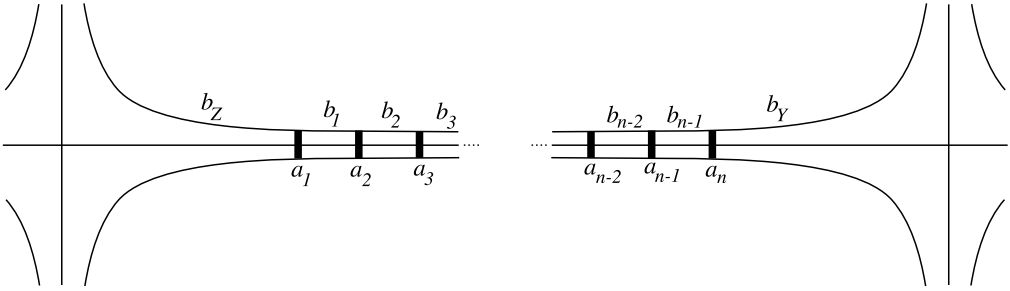}}
\caption{Leaf between two singular points}
\label{crosstocross}
\end{figure}

If we look at the coboundary conditions for this part of the picture,
we will get a system of equations like
\begin{equation}\label{longcohom}
\begin{split}
a_1 &= b_1 - b_Z\\
a_2 &= b_2 - b_1\\
&\vdots\\
a_n &= b_Y - b_{n-1}
\end{split}
\end{equation}
(where for simplicity we have omitted the terms giving parallel transport).
Adding up all these equations gives
\begin{equation}
a_1 + a_2 + \cdots + a_n = b_Y - b_Z
\end{equation}
which, since $b_Y$ and $b_Z$ must be analytically flat,
shows that the sum of the $a_j$'s must be analytically flat.

The part of the cohomology coming from this part of the picture
will therefore have a term of the form
\begin{equation}
\frac{C^\infty(I)^{\oplus n}}{\{ a_1 + \cdots + a_n \approx 0 \} },
\end{equation}
which is isomorphic to $\C^\N$ by a similar argument as in the proof
of Lemma~\ref{CN}. Thus, the $H^1$ cohomology will have one
term of the form $\C^\N$ coming from this part of the singular leaf.

This will be true for each arc connecting two singularities in the
singular leaf, and these conditions will be independent of each other.
Since there are twice as many such arcs as singular points
(four emitting from each point, each of which gets counted twice this way),
there are two $\C^\N$ contributions in per singularity.

For the same reason as in the proof of Theorem~\ref{higher-cohom} 
(namely that the covering has no triple or higher intersections), the 
higher cohomology groups are zero.  
From~\cite{mhthesis}, we have a Mayer-Vietoris principle for this 
cohomology (see Propositions 3.4.2 and 6.3.1).  
Putting together the results of this section 
with the results from~\cite{sniatpaper}
and~\cite{mhthesis} (which give the regular and elliptic cases, respectively),
and patching together with Mayer-Vietoris, we obtain the following:

\begin{thm}\label{mainthm}\label{thmn:surfaces}
Let $(M, \omega, F)$ be a two-dimensional, compact, completely integrable
system, whose moment map has only nondegenerate singularities.
Suppose $M$ has a prequantum line bundle $\LL$,
and let $\J$ be the sheaf of sections of $\LL$ flat along the leaves.
The cohomology $H^1(M,\J)$ has two
contributions of the form $\C^\N$ for each hyperbolic singularity,
each one corresponding to a space of Taylor series in one complex variable,
and one $\C$ term for each non-singular Bohr-Sommerfeld leaf.
That is,
\begin{equation}\label{eq:mainthm}
H^1 (M;\J) \cong \bigoplus_{p \in \mathcal{H}}
\bigl( \C^\N \oplus \C^\N\bigr)
\oplus \bigoplus_{b\in BS} \C_b .
\end{equation}
The cohomology in other degrees is zero.

Thus in particular,
the quantization of $M$
is given by~\eqref{eq:mainthm}.
\end{thm}

\begin{rmkn}
So far, we have only shown the above for cohomology computed
with respect to the particular coverings used in the computations,
but we prove below in~\S\ref{refcov} that this is isomorphic to 
the actual sheaf cohomology.
\end{rmkn}

\section{ Dependence on polarizations}\label{s:hopf}

The theorem above establishes a strong dependence of the
quantization of an integrable system on a surface
on the singularities of the function
determining the integrable system. In particular, if
we can find examples of  integrable systems on the same surface with
different kinds of singularities, Theorem~\ref{mainthm}
would show that this notion of
quantization depends strongly on the polarization considered.

\subsection{Two examples from Mechanics}

In this section we give two examples which
show up naturally in mechanics and then we give a method to
construct general examples of surfaces with prescribed number of
hyperbolic singularities.

\gap

\noindent{\bf{Example 1: Rotations on the sphere}}

Consider the height function $h$ on the 2-sphere of integer height
$k$  together with its standard area form. The Hamiltonian vector
field of the function $h$ is the vector field given by
rotations along the central axis.

As described in \cite{mhthesis} Chapter 5, this system has $k-1$
non-singular Bohr-Sommerfeld leaves, corresponding to the circles with integer
height. A picture of the integrable system with the Bohr-Sommerfeld
leaves marked on it for $k=4$ is shown in Figure~\ref{height}.

\begin{center}
\begin{figure}[h0]
\centerline{\includegraphics{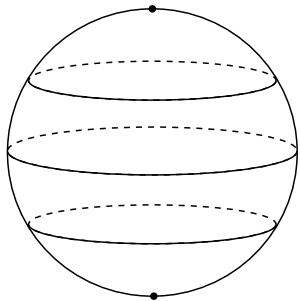}} \caption{The height function
on a sphere}\label{height}
\end{figure}
\end{center}

According to theorem \ref{thmn:surfaces}, the dimension of the
quantization for this integrable system is just given by the regular
Bohr-Sommerfeld leaves, which in this case is $k-1$. The elliptic
singularities (north and south poles) do not contribute.

\gap

\noindent{\bf Example 2: Euler's equations restricted to a sphere}

Consider Euler's equations of the rigid body on $T^*(SO(3))$ and
consider the lifted action of $SO(3)$. These equations correspond to
the movement of the Euler top (a rigid body moving around its center
of mass) which has configuration space $SO(3)$. Using symplectic
reduction by the lifted action of $SO(3)$ we obtain a Hamiltonian
system on $S^2$.
 The topology and geometry of the
induced system on the symplectic reduced space  is well-known; see
for example Cushman and Bates \cite{batesandcushman} for details.
 In section III.4, they show that this system has 
two hyperbolic singularities and four elliptic 
singularities.

A picture of the integrable system is given in Figure~\ref{euler}.

\begin{center}
\begin{figure}[h4]
\centerline{\includegraphics{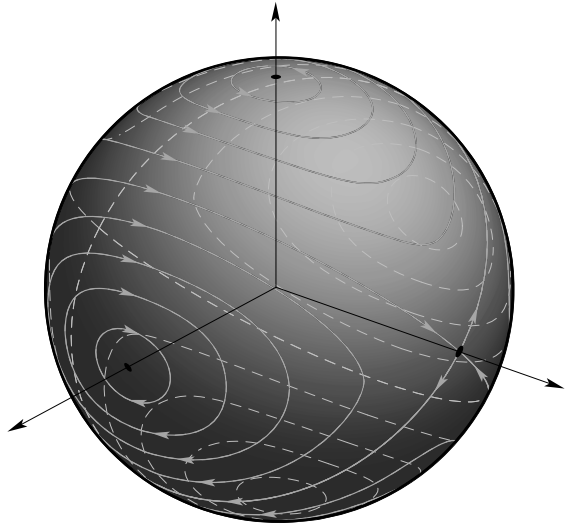}} \caption{Reduced Hamiltonian
flow of Euler's equations}\label{euler}
\end{figure}
\end{center}

Using the recipe given in Theorem \ref{thmn:surfaces}, the
quantization of this system is

\begin{equation*}
\mathcal{Q}(M) = H^1 (M;\J) \cong \bigoplus_ {p\in \mathcal{H}}
(\C_p^\N)^2 \oplus \bigoplus_{b\in BS} \C_b .
\end{equation*}

Since the hyperbolic set has two elements, this cohomology group has
four infinite-dimensional contributions.  If we compare this
example to the previous one (in which the quantization is
finite-dimensional), we can conclude that this quantization of the
sphere strongly depends on the polarization when we allow
singularities.

\subsection{Surgery of integrable systems}
Indeed, we can perform surgery of integrable systems to include as
many hyperbolic singularities into the picture as the Euler
characteristic allows. We briefly present this method in this small
subsection, for the sake of completeness.
Though the construction might seem elementary, such an explicit
description is not detailed in the literature of
integrable systems.

Given a function on a compact orientable surface $f\colon
S\longrightarrow \mathbb R$ with non-degenerate singularities (a
Morse function), consider the Hamiltonian vector field $X_f$
associated to this function. It is well-known (see for instance,
\cite{milnor}) that the number of elliptic and hyperbolic
singularities of this vector field on a surface is related to the
Euler characteristic via the Poincar\'{e}-Hopf formula:
\begin{equation}\label{morse}
\chi(S)=\text{number of elliptic singularities} - \text{number of
hyperbolic singularities}
\end{equation}

In the case of compact orientable surfaces, we can find examples of
integrable systems on them with any numbers $s_e$ of elliptic
and $s_h$ hyperbolic
singularities greater than for the height function and
satisfying~\eqref{morse}.
These examples can be created via surgery of integrable systems, 
adding cylinders with one elliptic and
one hyperbolic singularity and therefore increasing by one the
number of each type of singularity at each step.
Bolsinov and Fomenko \cite{bolsinovandfomenko} have developed a
whole Morse theory for integrable systems of singularities
with special attention to the cases of surfaces.

We denote $s_e^0$ and $s_h^0$  the total number of
elliptic and hyperbolic singularities given by the height function
on the compact surface.

The method has the following steps, which we illustrate on the sphere
in the figures below.\footnote{We wish to thank
Alexey Bolsinov for clarifying this procedure to us
in Oberwolfach during the finishing stages of work on
this paper.}

\gap

\noindent{\bf Step 1:} Start with the height function $h$ on a given
compact surface. Cut out a cylinder containing only regular levels
following the level sets of the height function $h$. The upper and
lower border of the cylinder are level sets of $h$.
(Figure~\ref{cylcut})

\begin{figure}
\centerline{  \includegraphics[scale=0.3]{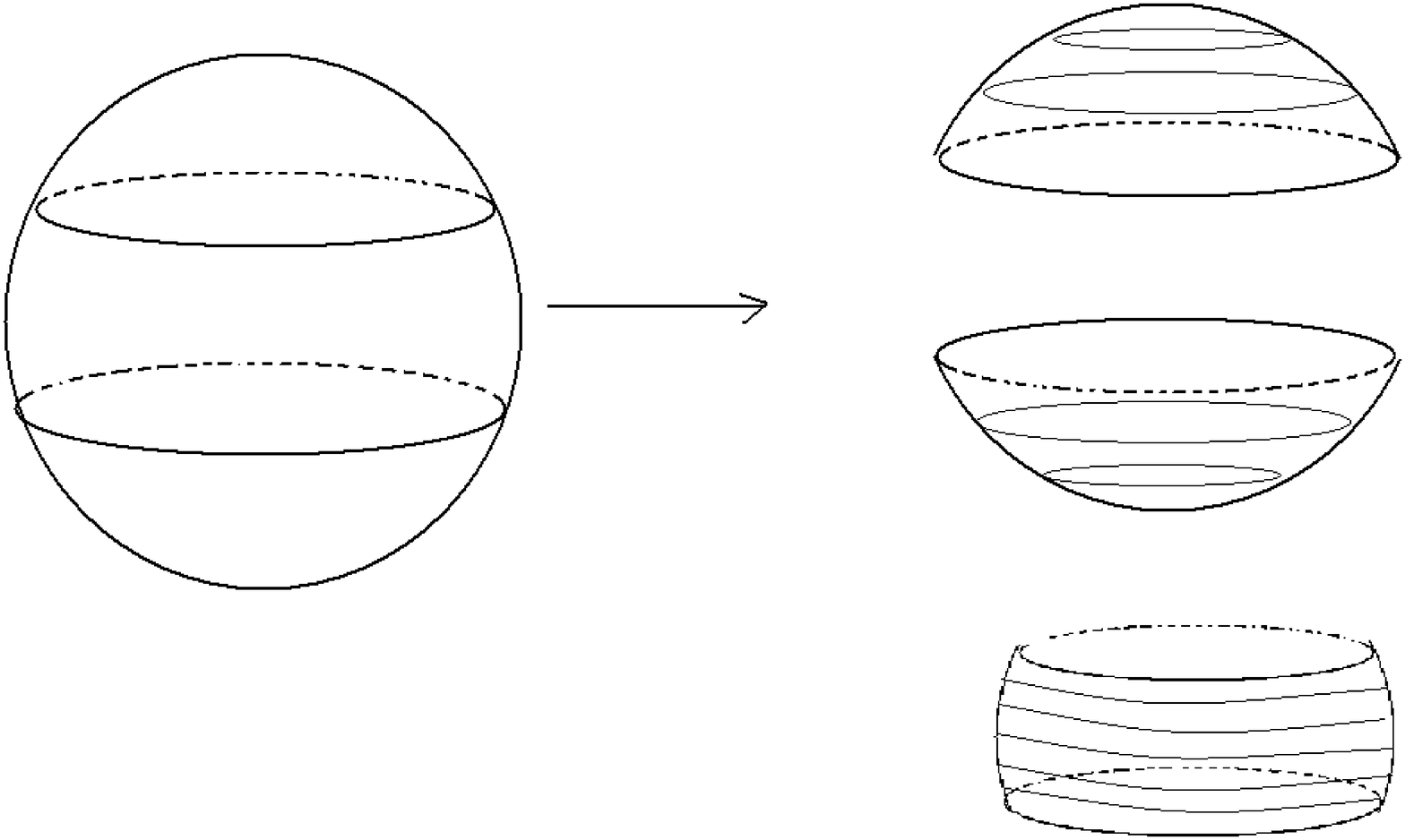} }
\caption{Cutting out a cylinder.}\label{cylcut}
\end{figure}

\gap

\noindent{\bf Step 2:} Leaving the foliation by level sets of $h$
the same on the complement of the cylinder, change the function
inside the cylinder (which is regular) in such a way as to create a
hyperbolic singularity and simultaneously an elliptic singularity
inside the cylinder.  See Figure~\ref{singcyl}.

\begin{figure}
\centerline{\includegraphics[scale=0.35]{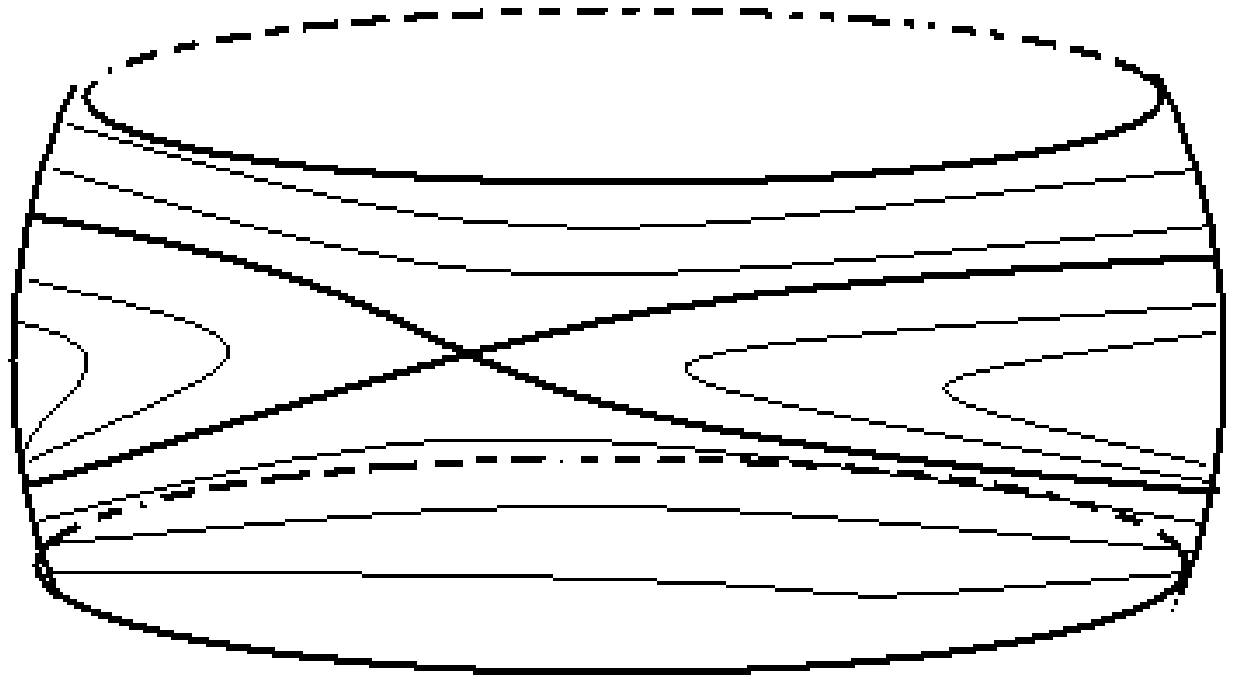} \quad
\quad \quad \includegraphics[scale=0.35]{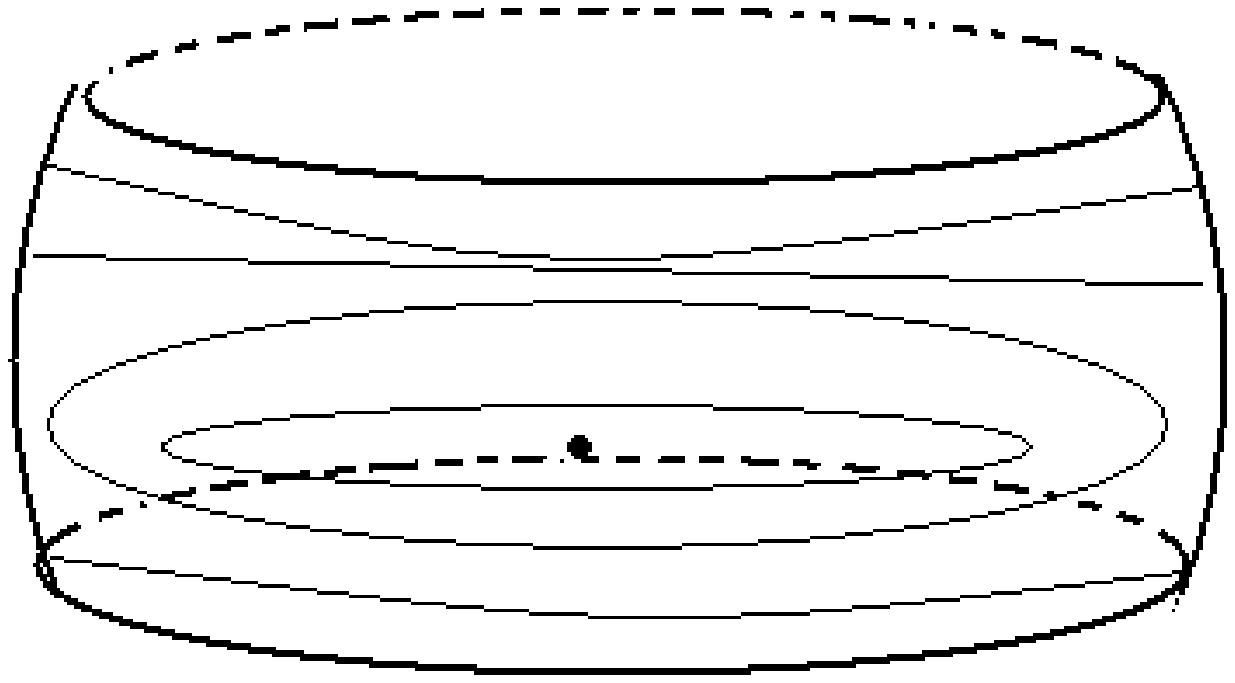}}
\caption{Hyperbolic singularity inside the cylinder. (Front and rear view)}
\label{singcyl}
\end{figure}

\gap

\noindent{\bf Step 3:} Glue the cylinder back into the surface. This
gives an example of a ``modified'' integrable system with one more
elliptic and one more hyperbolic singularity than we started with.

\gap

Finally, given $s_e$ and $s_h$
such that $\chi(S)= s_e-s_h$ and $s_e\geq s_e^0$ and $s_h\geq s_h^0$,
by repeating this process we can obtain an
example of an integrable system on a compact surface with exactly
$s_e$ elliptic singularities and $s_h$ hyperbolic singularities.

For these systems, we
can apply the main recipe of theorem \ref{thmn:surfaces} to get the
following result:

\begin{prop} The quantization of the integrable system constructed above via
integrable surgery on a compact orientable surface with Euler
characteristic $\chi(S)=s_e-s_h$ and exactly $s_e$ elliptic and
$s_h$ hyperbolic singularities  such that  $s_e\geq s_e^0$ and
$s_h\geq s_h^0$ is given by the formula:
\[ H^1 (M;\J) \cong \mathbb (C^\N)^{2s_h} \oplus
\bigoplus_{b\in BS} \C_b . \]
\end{prop}

This shows that this quantization of any compact surface strongly
depends on the polarization when we allow polarizations with
singularities.

\section{Refinements and coverings}\label{refcov}

In this somewhat technical section we show that the cohomology computed in
sections~\ref{cohom-setup}~--~\ref{s:manysing}
is the actual sheaf cohomology.
We use the methods and terminology of~\cite{mhthesis},
especially \S3.4 and 3.5.  We review these briefly and refer the reader
there for more details.

Let $M$ be a compact 2-dimensional prequantized integrable system, as
usual.
Recall that sheaf cohomology is defined as the direct limit, over
all open coverings of $M$, of the cohomology computed
with respect to the cover.
In order to show that the cohomology we have computed in~\S\ref{ugly} is the
actual sheaf cohomology, we show that every open covering has a refinement
whose cohomology is isomorphic to that found in~\S\ref{ugly}.
For simplicity, we assume that $M$ has only one leaf with hyperbolic
singularities; the extension to the case of several such leaves is
reasonably straightforward.

We copy from~\cite{mhthesis} the following definition.
We assume we have a given set of coordinates
(which will usually be action-angle coordinates),
which we call  $(t,\theta)$.
\begin{defn}
A \define{brick wall} cover of a $t$-$\theta$ rectangle
is a finite covering by open $t$-$\theta$ rectangles (``bricks''),
satisfying the following properties:
\begin{itemize}
\item The rectangles can be partitioned into sets (``layers'') so that
all rectangles in one layer cover the same interval of $t$ values
(``All bricks in the same layer have the same height'');
\item Each brick contains points that are not in any other brick; and
\item There are no worse than triple intersections, i.e., the
intersection of two bricks in one layer does not meet the intersection
of two bricks in either of the two adjoining layers.
\end{itemize}
\begin{figure}[htbp]
\centerline{\includegraphics[width=1.5in,height=1in]{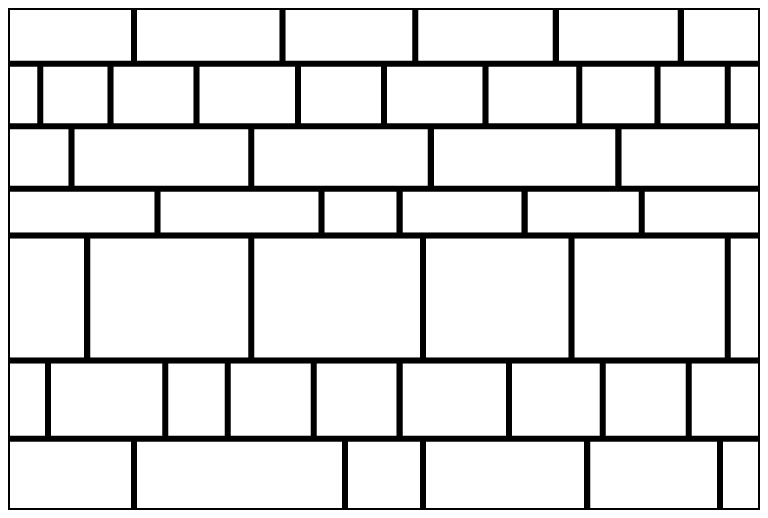}
\qquad \qquad \includegraphics[width=1.5in,height=1in]{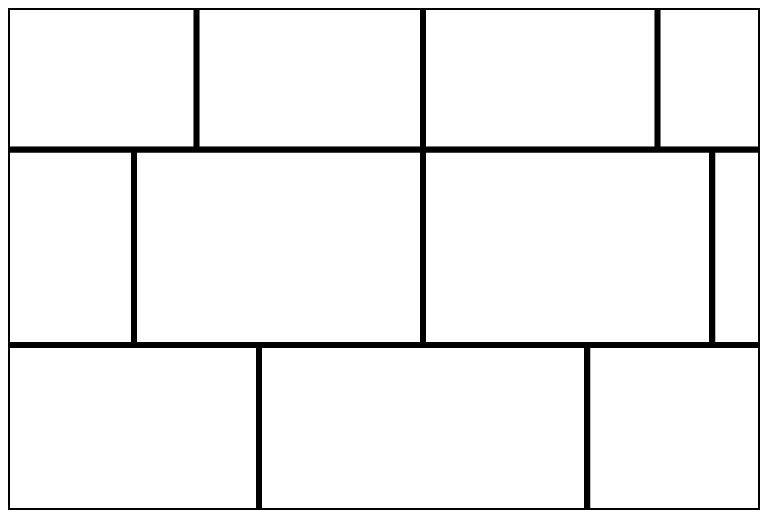}}
\caption{A brick wall cover, and one which is not}\label{brickfig}
\end{figure}
Note that we do not require that the number of bricks be the same in each
layer, nor that the layers have the same height, nor that the bricks within
one layer have the same width.
See Figure~\ref{brickfig}, where thick lines indicate intersections.
The definition extends in the obvious way to cylinders, where we identify
$\theta=0$ and $\theta=2\pi$.
\end{defn}

Let $\mathfrak{U}$ be an open covering of $M$.
By Lebesgue's number lemma, there is some number $\delta$ such that
any set of diameter less than $\delta$ is contained entirely in
some set $U \in \U$.

Let $L_\delta$ be a ``fattening'' of the singular leaf,
a neighbourhood of the singular leaf which is the union of leaves
of width $\delta/2$.  Cover $L_\delta$ by rectangles of width $\delta/2$,
together with a small ``hyperbolic cross'' at the singularity that
also has diameter less than $\delta$.
Then let the open covering $\V$ be the collection of these rectangles,
together with a brick wall covering of $M\smallsetminus L_\delta$ with
bricks of diameter less than $\delta$.
Then $\V$ is a refinement of $\U$.

Now we show that the \v Cech cohomology of $M$ calculated with respect
to $\V$ is the same as we found in~\ref{ugly}.

Let $E$ be the union of all layers of bricks which do not meet
the singular leaf.
Let $A\subset L_\delta$
be an open union of leaves around the singular leaf which does not
intersect $E$ and which does not contain any Bohr-Sommerfeld leaf
other than possibly the singular leaf.
(This is possible by the discreteness of Bohr-Sommerfeld leaves.)
Let $B$ be an open union of regular leaves such that $A\cup B = M$.
Then the covering $\V$ induces a covering on $A$ and $B$ which is a
brick wall covering on $B$, and on $A$ has the same form as shown 
in Figure~\ref{leafcov}.

(The point of this construction is the following: 
$A$ meets only one layer of bricks,
namely the ones covering the singular leaf.
$B$ is an open union of leaves which, together with $A$, covers $M$.
We have chosen $A$ and $B$ so that all intersections between layers
of bricks happen \emph{outside} of $A$.
This means that the covering induced on $A$ has no triple intersections, 
and we can apply the results of~\S\ref{ss:generalcase}.
Since we have from~\cite{mhthesis} a Mayer-Vietoris for 
unions of regular leaves, and since $A\cap B$ consists only of 
regular leaves, we can apply Mayer-Vietoris to $A$ and $B$.
Thus we avoid having to calculate with a covering of $A$ with more 
``layers of bricks,'' and thus avoid triple intersections.)

By the assumption that $A$ contains no Bohr-Sommerfeld leaf, the
cohomology of $A$ with respect to the covering induced by $\V$ is
$\C^{2\N}$ in degree 1, and zero otherwise.
Since $B$ is regular, the results of~\cite{mhthesis} apply, and the
cohomology of $B$ with respect to the covering induced by $\V$
has one dimension for every (nonsingular) Bohr-Sommerfeld leaf.

By Mayer-Vietoris, $H^*_\V(M) \cong H^*_\V(A) \oplus H^*_\V(B)$
since $A\cap B$ is regular and has no Bohr-Sommerfeld leaves.

Therefore, the cohomology of $M$ calculated with respect to the covering
$\V$ is the same as that calculated with respect to the covering
in~\S\ref{cohom-setup}.

Since every open covering $\U$ has a refinement of the form $\V$, 
we have computed the actual sheaf cohomology of $M$.

\end{document}